\documentclass[11pt]{article}
\usepackage{amsmath}
\usepackage{dsfont}
\usepackage{mathrsfs}
\usepackage{amsmath,amssymb}
\usepackage{amsfonts}
\usepackage{hyperref}
\usepackage{amsthm}
\usepackage{graphicx}
\usepackage{subfigure}
\usepackage{xcolor}
\renewcommand{\qed}{\hfill\small{$\square$}\normalsize}

\hfuzz=\maxdimen
\theoremstyle{definition}
\newtheorem{lemma}{Lemma}[section]
\newtheorem{definition}[lemma]{Definition}
\newtheorem{proposition}[lemma]{Proposition}
\newtheorem{theorem}[lemma]{Theorem}

\newtheorem{example}{Example}
\newtheorem{remark}{Remark}

\numberwithin{equation}{section}

\renewcommand{\qed}{\hfill\small{$\square$}\normalsize}

\DeclareFixedFont{\Acknowledgment}{OT1}{cmr}{bx}{n}{14pt}
\begin{document}

\title{\bf Unified $(\alpha,\beta)$-Flows on Triangulated Manifolds with Two and Three Dimensions}
\author{Huabin Ge, Ming Li}
\maketitle

\begin{abstract}
In this paper, we introduce a framework of $(\alpha,\beta)$-flows on triangulated manifolds with two and three dimensions, which unifies
several discrete curvature flows previously defined in the literature.
\end{abstract}

\footnotetext{{\it running head.} Unified $(\alpha,\beta)$-Flows on Triangulated Manifolds}

\section{Background and Preliminaries}\label{Introduction}
In his famous book \cite{T1}, Thurston introduced the circle packing metric, which was used to study low dimensional topology. In \cite{Ha1},
Hamilton introduced the well-known Ricci flow, which is a powerful tool to deform Riemannian metrics. More important to our subject, as a pioneering work, Chow and Luo \cite{CL1} established the intrinsic connection between Hamilton's surface Ricci flow and Thurston's circle patterns. They first introduced the combinatorial (discrete) Ricci flows on triangulated surface. Inspired by their work, Glickenstein \cite{G1} first introduced the combinatorial Yamabe flows on triangulated 3-dimensional manifolds. Since then, discrete curvature flow has been becoming popular for its usefulness in engineering fields, especially in the Graphics and Image Processing areas.

Suppose $M^n$ is a closed manifold with dimension $n=2$ or $3$. Given a triangulation $\mathcal{T}=\{\mathcal{T}_0,\mathcal{T}_1,\cdots,\mathcal{T}_n\}$ on $M$, where $\mathcal{T}_i$ represents the set of $i$-simplices ($0\leq i\leq n$). In what follows, $(M^n, \mathcal{T})$ will be referred to as a triangulated $n$-manifold. All the vertices are ordered one by one, marked by $v_1, \cdots, v_N$, where $N=\mathcal{T}_0^\sharp$ is the number of vertices. We use $i\sim j$ to denote that the vertices $i$ and $j$ are adjacent if there is an edge $\{ij\}\in\mathcal{T}_1$. Throughout this paper, all functions $f: \mathcal{T}_0\rightarrow \mathds{R}$ will be regarded as column vectors in $\mathds{R}^N$. And we denote $C(\mathcal{T}_0)$ as the set of functions defined on $\mathcal{T}_0$.

The most natural way to define a discrete metric on a triangulated manifold $(M^n, \mathcal{T})$ is to evaluate a length $l_{ij}$ for each edge $i\sim j$ directly. Alternatively,
discrete metrics can also be defined on all vertices, based on which the length $l_{ij}$ can be derived indirectly. Thurston introduced the circle packing metric for $n=2$, while Cooper and Rivin considered the sphere packing metric for $n=3$. Now we review these two definitions:

\begin{definition}
(Thurston's circle packing metric) Given a triangulated surface $(M^2, \mathcal{T})$. Let $\Phi: \mathcal{T}_1\rightarrow [0,\frac{\pi}{2}]$ be a function assigning each edge $\{ij\}$ a weight $\Phi_{ij}$. Each map $r:\mathcal{T}_0\rightarrow (0,+\infty)$ is called a circle packing metric.
\end{definition}

For given $(M^2, \mathcal{T}, \Phi)$, we attach each edge $\{ij\}$ a length
\begin{equation}\label{definition of length of edge}
l_{ij}=\sqrt{r_i^2+r_j^2+2r_ir_j\cos \Phi_{ij}}.
\end{equation}
Thurston proved \cite{T1} that the lengths $\{l_{ij}, l_{jk}, l_{ik}\}$ satisfy the triangle inequality
on each face $\{ijk\}\in \mathcal{T}_2$, which ensures that the face $\{ijk\}$ could be
realized as an Euclidean triangle with lengths $\{l_{ij}, l_{jk}, l_{ik}\}$. Thus the space of all circle packing metrics is exactly $$\mathds{R}^N_{>0}\triangleq (0,+\infty)^N.$$

The triangulated surface $(M^2, \mathcal{T}, \Phi)$ can be seen as
gluing many Euclidean triangles coherently. However, this gluing produces singularities at vertices, which is described as discrete curvature.
\begin{definition}(Discrete Gaussian curvature)
Suppose $\theta_i^{jk}$ is the inner angle of triangle $\triangle v_iv_jv_k$
at vertex $i$, the classical discrete Gauss curvature at $i$ is defined as
\begin{equation}\label{classical Gauss curv}
K_i=2\pi-\sum_{\{ijk\}\in \mathcal{T}_2}\theta_i^{jk},
\end{equation}
where the sum is taken over all the triangles with $v_i$ as one of their vertices.
\end{definition}
For discrete Gaussian curvature $K_i$, there is a
discrete version of Gauss-Bonnet identity
\begin{equation}\label{Gauss-Bonnet}
\sum_{i\in \mathcal{T}_0}K_i=2\pi \chi(M).
\end{equation}

\begin{definition}(Cooper $\&$ Rivin's ball packing metric)
Given a triangulated surface $(M^3, \mathcal{T})$. Each map $r:\mathcal{T}_0\rightarrow (0,+\infty)$ derives a length
\begin{equation}
l_{ij}=r_{i}+r_{j}
\end{equation}
for each $i\thicksim j$. If for each $\{i,j,k,l\}\in \mathcal{T}_3$, $l_{ij},l_{ik},l_{il},l_{jk},l_{jl},l_{kl}$
determines an Euclidean tetrahedron, then call $r$ a (non-degenerate) ball packing metric.
\end{definition}

It is pointed out \cite{G1} that a tetrahedron $\{i,j,k,l\}\in \mathcal{T}_3$ generated by four positive radii $r_{i},r_{j},r_{k},r_{l}$
can be realized as an Euclidean tetrahedron if and only if
\begin{equation}\label{nondegeneracy-condition}
Q_{ijkl}=\left(\frac{1}{r_{i}}+\frac{1}{r_{j}}+\frac{1}{r_{k}}+\frac{1}{r_{l}}\right)^2-
2\left(\frac{1}{r_{i}^2}+\frac{1}{r_{j}^2}+\frac{1}{r_{k}^2}+\frac{1}{r_{l}^2}\right)>0.
\end{equation}
Thus the space of all ball packing metrics is exactly
$$\mathfrak{M}_{\mathcal{T}}=\left\{\;r\in\mathds{R}^N_{>0}\;\big|\;Q_{ijkl}>0, \;\forall \{i,j,k,l\}\in \mathcal{T}_3\;\right\}.$$
Cooper and Rivin \cite{CR} proved that $\mathfrak{M}_{\mathcal{T}}$ is a simply connected open subset of $\mathds{R}^N_{>0}$, but not convex.
The ball packing metric induces a piecewise linear metric which makes the curvature flat everywhere on $M^{3}-\mathcal{T}_1$ while singular on $\mathcal{T}_1$. Thus Cooper and Rivin defined a discrete scalar curvature concentrated on all vertices.
\begin{definition}(Cooper $\&$ Rivin's discrete scalar curvature)
Denote $\alpha_{ijkl}$ as the solid angle at a vertex $i$ in a single Euclidean tetrahedron $\{i,j,k,l\}\in \mathcal{T}_3$, then the discrete scalar curvature at vertex $i$ is
\begin{equation}\label{def-cr-curv}
K_{i}=4\pi-\sum_{\{i,j,k,l\}\in \mathcal{T}_3}\alpha_{ijkl},
\end{equation}
where the sum is taken over all $\{j,k,l\}\in \mathcal{T}_2$ such that $\{i,j,k,l\}\in \mathcal{T}_3$.
\end{definition}
In what follows we call Cooper $\&$ Rivin's discrete curvature as CR-curvature for short. To study the CR-curvature $K_i$, the first of the authors and Xu \cite{GX3} introduced the following $3$-dimensional $\alpha$-flow,
\begin{equation}\label{def-alpha-flow}
\frac{dr_i}{dt}=s_{\alpha}r_i^{\alpha}-K_i,
\end{equation}
where $s_{\alpha}=\frac{\sum_i K_ir_i}{\sum_ir_i^{\alpha+1}}$. They proved that if the $\alpha$-flow (\ref{def-alpha-flow}) converges, there exists a metric $r^*$ whose $\alpha$-order curvature $K_i/r_i^{\alpha}$ is a constant for each $i\in V$. On the contrary, assume $r^*$ is a metric whose $\alpha$-order curvature is a constant, and the first positive eigenvalue of $-\Delta_{\alpha}$ (see (\ref{Def-alpha-Laplacian}) for a definition) at $r^*$ is bigger than $\alpha s_{\alpha}^*$, then $r^*$ is a asymptotically stable point of the $\alpha$-flow (\ref{def-alpha-flow}). In this paper, we shall generalize the $\alpha$-flow (\ref{def-alpha-flow}) to one with the following more universal form
$$\frac{dg_i}{dt}=s_{\alpha}r_i^{\alpha}-K_i,$$
where $g_i=\ln r_i$ or $g_i=r_i^\sigma$, $\sigma\in\mathds{R}$. We explain the idea behind this generalization.
If we consider the conformal deformation of the discrete metric $r_i$, we may take $g_i=\ln r_i$. This is inspired by Chow and Luo's pioneer work \cite{CL1}. If we consider the deformation of the conical metric, we may take take $g_i=r_i^{\sigma}$ as a metric (of $\sigma$-order). This is inspired by the viewpoint of Riemannian geometry. A piecewise flat metric is a singular Riemannian metric on $M$, which produces conical singularities at all vertices. For any $\sigma\in\mathbb{R}$, a metric $g$ with conical singularity at a point can be expressed as
\begin{equation}\label{def-universal-flow}
g(z)=e^{f(z)}\frac{dz\wedge d\bar{z}}{|z|^{2(1-\sigma)}}
\end{equation}
locally. Choosing $f(z)=-\ln\sigma^{2}$, then $g(z)=|dz^{\sigma}|^{2}$. Comparing $r^{\sigma}$ with $|dz^{\sigma}|$, the $\sigma$-order metric
$r^{\sigma}$ may be considered as a discrete analogue of conical metric to some extent. In this paper, we will express the flow (\ref{Def-unified-flow}) as $\frac{dr_i}{dt}=s_{\alpha}r_i^{\beta}-K_ir_i^{\beta-\alpha}$ by leaving out a constant $\sigma$ and taking $\beta=1+\alpha-\sigma$ (see Definition \ref{def-above-flow} below). We will prove that the flow (\ref{Def-unified-flow}) exhibits similar convergence properties as the $\alpha$-flow (\ref{def-alpha-flow}) in section \ref{section-convergence}.

\section{Definition of unified $(\alpha,\beta)$-flow}\label{Section-unified flow}
For any $\alpha \in \mathds{R}$, the first of the authors and Xu \cite{GX2,GX3} studied discrete circle (ball) packing metrics with all curvature $K_i/r_i^{\alpha}$ equal to a constant. It is easy to see that, if $\frac{K_i}{r_i^{\alpha}}\equiv s_{\alpha}$ for each vertex $i$, where $s_{\alpha}$ is a constant, then
\begin{equation}\label{def-s-alpha}
s_{\alpha}=\frac{\sum K_ir_i^{n-2}}{\|r\|_{\alpha+{n-2}}^{\alpha+{n-2}}}.
\end{equation}
Following the definition in \cite{GM2}, we call this type of metric as ``discrete $\alpha$-quasi-Einstein metric" or ``constant $\alpha$-curvature metric".

One of the most important problem we concern is to understand whether there are discrete $\alpha$-quasi-Einstein metrics in $\mathfrak{M}_{\mathcal{T}}$ for $n=3$, or in $\mathds{R}^N_{>0}$ for $n=2$. For the $n=2$ case, this problem is basically resolved in \cite{GX2,GX3}. Especially for triangulated surfaces with $\alpha\chi(M)\leq0$, both the combinatorial-topological conditions and the analytical conditions are given for the existence of discrete $\alpha$-quasi-Einstein metrics. Furthermore, discrete $\alpha$-quasi-Einstein metrics can be obtained by evolving discrete curvature flows or minimizing discrete Ricci potentials. For the $n=3$ case, very few combinatorial-topological conditions are known for the existence of discrete $\alpha$-quasi-Einstein metrics. We want to study discrete $\alpha$-quasi-Einstein metrics by introducing an unified $(\alpha,\beta)$-flow:
\begin{definition}\label{def-above-flow}
Let $s_{\alpha}$ be defined in (\ref{def-s-alpha}). If $n=2$, let $K_i$ be defined in (\ref{classical Gauss curv}). If $n=3$, let $K_i$ be defined in (\ref{def-cr-curv}). For any $\alpha, \beta\in \mathds{R}$, the unified $(\alpha,\beta)$-flow is defined as
\begin{equation}\label{Def-unified-flow}
\frac{dr_i}{dt}=s_{\alpha}r_i^{\beta}-K_ir_i^{\beta-\alpha}.
\end{equation}
\end{definition}

The prototype of the unified $(\alpha,\beta)$-flow (\ref{Def-unified-flow}) is Chow-Luo's combinatorial Ricci flow \cite{CL1} in dimension two, which can be expressed as $\frac{dr_i}{dt}=K_{av}r_i-K_ir_i$. Take $(\alpha,\beta)=(0,1)$, and $n=2$, then the unified $(\alpha,\beta)$-flow (\ref{Def-unified-flow}) becomes Chow-Luo's flow. The unified $(\alpha,\beta)$-flow (\ref{Def-unified-flow}) modeled its form after that of the well known smooth Yamabe flow $\frac{\partial g}{\partial t}=(s-R)g$, where $g$ is a smooth Riemannian metric tensor, $R$ is the scalar curvature and $s$ is the average curvature. If we express the unified $(\alpha,\beta)$-flow (\ref{Def-unified-flow}) as $\frac{dr_i}{dt}=(s_{\alpha}-\frac{K_i}{r_i^{\alpha}})r_i^{\beta}$ and further take $\beta=1$, then it's easy to see that the term $\frac{K_i}{r_i^{\alpha}}$ plays similar role as the term $R$ plays in the smooth Yamabe flow. To some extent, $\frac{K_i}{r_i^{\alpha}}$ is the discrete analogy of the smooth scalar curvature $R$.

Now it's the time to say a bit more about the motivations to introduce the unified $(\alpha,\beta)$-flow. The first motivation is to unify the various discrete curvature flows previously defined in the literature, see Section \ref{Section-examples} for more details. The second motivation is
to approach the following combinatorial $\alpha$-Yamabe problem, which is modeled after the smooth Yamabe problem and was previously raised by the first of the authors and Xu \cite{GX3}.\\

\noindent\textbf{The Combinatorial $\alpha$-Yamabe Problem.}
Given a $2$-dimensional (or $3$-dimensional) manifold $M^2$ (or $M^3$) with triangulation $\mathcal{T}$,
find a circle (or ball) packing metric with constant combinatorial $\alpha$-curvature in the combinatorial
conformal class $\mathds{R}_{>0}^N$ (or $\mathfrak{M}_{\mathcal{T}}$).\\

The unified $(\alpha,\beta)$-flow provides a natural way to approach the combinatorial $\alpha$-Yamabe problem. We want to deform an arbitrary metric $r(0)$ to a discrete $\alpha$-quasi-Einstein metric $r^*$, one effective way is to evolve it according to an ODE system $r_i'(t)=f_i(r(t))$ with $r^*$ as its critical point. Thus $f_i(r^*)$ equals to zero. The easiest way is to choose $f_i(r)=s_{\alpha}-\frac{K_i}{r_i^{\alpha}}$. In fact, we found this selection of $f_i(r)$ is too restrictive, this fact motivates us to relax it as  $f_i(r)=s_{\alpha}r_i^{\alpha}-K_ir_i^{\beta-\alpha}$. In the following Theorem \ref{them-conveg-imply-solvable}, we will show that if the solution to the $(\alpha,\beta)$-flow converges, then the combinatorial $\alpha$-Yamabe problem is solvable.

We shall show that the unified $(\alpha,\beta)$ flow (\ref{Def-unified-flow}) is generally not a ``normalization" of the following flow
\begin{equation}\label{gejiang's-flow}
\frac{dr_i}{dt}=-K_ir_i^{\beta-\alpha}
\end{equation}
except for $(\alpha,\beta)=(0,1)$, where the word ``normalization" means that the solutions to this two flows are related by a change of the scale in space and a change of the parametrization in time. Assume the solution
of the unified $(\alpha,\beta)$ flow (\ref{Def-unified-flow}) differs from the solution of the flow (\ref{gejiang's-flow}) only by a change of the scale in space and a change of the parametrization in time. Let $t,r,K,s_{\alpha}$ denote the variables for the flow (\ref{gejiang's-flow}), and $\tilde{t},\tilde{r},\tilde{K},\tilde{s}_{\alpha}$ for the flow (\ref{Def-unified-flow}). Let $r(t),t\in[0,T)$ be a solution of the flow (\ref{gejiang's-flow}), while $\tilde{r}(\tilde{t}),\tilde{t}\in[0,\tilde{T})$ be a solution of the flow (\ref{Def-unified-flow}). Set $\tilde{r}(\tilde{t})=\varphi(t)r(t)$, where $\varphi(t)>0$ is a scaling factor and is independent of the vertices. Obviously, $\tilde{K}(\tilde{t})=K(t)$, $\tilde{s}_{\alpha}(\tilde{t})=s_{\alpha}(t)\varphi^{-\alpha}(t)$, and $\tilde{s}_{\alpha}\tilde{r}_{i}^{\alpha}=s_{\alpha}r_{i}^{\alpha}$. Hence
$$\frac{d\tilde{r}_{i}}{d\tilde{t}}=\frac{d(\varphi r_i)}{dt}\frac{dt}{d\tilde{t}}=
\Big(r_i\frac{d\varphi}{dt}-\varphi\tilde{K_{i}}r_i^{\beta-\alpha}\Big)\frac{dt}{d\tilde{t}}.$$
On the other hand, $\tilde{r}$ satisfies the equation (\ref{Def-unified-flow}),
i.e.,
$$\frac{d\tilde{r}_{i}}{d\tilde{t}}=\tilde{s}_{\alpha}\tilde{r}_{i}^\beta-\tilde{K_{i}}\tilde{r}_{i}^{\beta-\alpha}=
s_{\alpha}r_{i}^{\beta}\varphi^{\beta-\alpha}-K_{i}r_{i}^{\beta-\alpha}\varphi^{\beta-\alpha}.$$
Comparing the above two expressions about $d\tilde{r}_{i}/d\tilde{t}$, we obtain $\beta-\alpha=1$, $dt/d\tilde{t}=1$ and $\frac{1}{\varphi}\frac{d\varphi}{dt}=s_{\alpha}r_{i}^{\alpha}$ for every $i\in V$. Since $\varphi$ is independent of $i\in V$, we further get $\alpha=0$ and $\beta=1$. Hence the unified $(\alpha,\beta)$ flow (\ref{Def-unified-flow}) is a normalization of the flow (\ref{gejiang's-flow}) if and only if $(\alpha,\beta)=(0,1)$.

\section{Some examples}\label{Section-examples}
There are several discrete curvature flows on two and three dimensional triangulated manifolds. We list some of them below. Firstly let us take a look at the 2-dimensional examples.
\begin{example}\label{Example-all-2d-flows}
Given a triangulated surface $(M^2, \mathcal{T})$, consider Thurston's circle packing metric $r$ with fixed weight $\Phi$. Denote the discrete $\alpha$-curvature as $R_{\alpha,i}=K_i/r_i^{\alpha}$. Notice that $s_{\alpha}=2\pi\chi(M)/\|r\|_{\alpha}^{\alpha}$ by discrete Gauss-Bonnet formula. Then we have the following six different discrete flows on triangulated surfaces:
\begin{description}
  \item[(1)] $\dot{u}_i=K_{av}-K_i$, where $u_i=\ln r_i$, $K_{av}=2\pi\chi(M)/N$, see \cite{CL1};
  \item[(2)] $\dot{u}_i=s_{\alpha}r_i^{\alpha}-K_i$, where $u_i=\ln r_i$, $\alpha\in \mathds{R}$, see \cite{GX3};
  \item[(2)$'$] $\dot{u}_i=s_{\alpha}r_i^{\alpha}-K_i$, where $u_i=\ln r_i^{\alpha}$, $\alpha\in \mathds{R}$ and $\alpha\neq0$, differs from (2) by a constant $\alpha$;
  \item[(3)] $\dot{g}_i=(R_{av}-R_{i})g_i$, where $g_i=r_i^2$, $R_{av}=s_2$ and $R_i=R_{2,i}$, see \cite{GX2};
  \item[(4)] $\dot{u}_i=s_{\alpha}-R_{\alpha,i}$, where $u_i=\ln r_i$, $\alpha\in \mathds{R}$, see \cite{GX2};
  \item[(4)$'$] $\dot{u}_i=s_{\alpha}-R_{\alpha,i}$, where $u_i=\ln r_i^{\alpha}$, $\alpha\in \mathds{R}$ and $\alpha\neq0$, differs from (4) by a constant $\alpha$.
\end{description}
\end{example}

Next we see the cases with dimension three.

\begin{example}\label{Example-all-3d-flows}
Given a triangulated 3-manifold $(M^3, \mathcal{T})$, consider Cooper and Rivin's sphere packing metric $r$ and $\alpha$-curvature $R_{\alpha,i}=K_i/r_i^2$. Then we have the following nine different discrete flows on $M^3$:

\begin{description}
  \item[(1)] $\dot{u}_i=-K_i$, where $u_i=\ln r_i$, see \cite{G1};
  \item[(2)] $\dot{r}_i=(s_0-K_i)r_i$, a normalization of the flow in (1), see \cite{GJ};
  \item[(3)] $\dot{r}_i=\lambda r_i-K_i$, with $\lambda=s_1$, see \cite{GX2};
  \item[(4)] $\dot{r}_i=s_{\alpha}r_i^{\alpha}-K_i$, $\alpha\in \mathds{R}$, see \cite{GX3};
  \item[(5)] $\dot{u}_i=s_{\alpha}-R_{\alpha,i}$ (or $\dot{r}_i=s_{\alpha}r_i-\frac{K_i}{r_i^{\alpha-1}}$), where $u_i=\ln r_i$, $\alpha\in \mathds{R}$, see \cite{GM1};
  \item[(5)$'$] $\dot{u}_i=s_{\alpha}r_i^{\alpha}-K_i$, where $u_i=r_i^{\alpha}$, $\alpha\in \mathds{R}$ and $\alpha\neq0$, differs from (5) by a constant $\alpha$;
  \item[(6)] $\dot{g}_i=(R_{av}-R_i)g_i$, where $g_i=r_i^2$, $R_{av}=s_2$, $R_i=K_i/r_i^2$, see \cite{GX2};
  \item[(7)] $\dot{u}_i=s_{\alpha}r_i^{\alpha}-K_i$, where $u_i=\ln r_i$, $\alpha\in \mathds{R}$, see \cite{GM2};
  \item[(7)$'$] $\dot{u}_i=s_{\alpha}r_i^{\alpha}-K_i$, where $u_i=\ln r_i^{\alpha}$, $\alpha\in \mathds{R}$ and $\alpha\neq0$, differs from (7) by a constant $\alpha$.
\end{description}
\end{example}

It is remarkable that all the discrete flows presented in Examples \ref{Example-all-2d-flows} and \ref{Example-all-3d-flows} can be unified by the $(\alpha,\beta)$-flow (\ref{Def-unified-flow}):

\begin{itemize}
  \item For the cases in Example \ref{Example-all-2d-flows}. If $\alpha=0$ and $\beta=1$, then the $(0, 1)$-flow is just the flow (1). If $\beta=\alpha+1$, then the $(\alpha, \alpha+1)$-flow is just the flow (2). If $\beta=1$, then the $(\alpha, 1)$-flow is just the flow (5). Notice that the flow (3) differs from the $(\alpha, \alpha+1)$-flow by a constant $\alpha$. Moreover, the flow (4) is a special case of the flow (6) with $\alpha=2$. The flow (6) differs from the $(\alpha, 1)$-flow by a constant $\alpha$.
  \item For the cases in Example \ref{Example-all-3d-flows}. If $\beta=\alpha=1$, then the $(1, 1)$-flow is just the flow (2). If $\beta=\alpha$, then the $(\alpha, \alpha)$-flow is just the flow (3). If $\beta=1$, then the $(\alpha, 1)$-flow is just the flow (4). If $\beta=\alpha+1$, then the $(\alpha, \alpha+1)$-flow is just the flow (7). Notice that the flow (1) is in fact a type of $(\alpha, \alpha+1)$-flow without normalization. The flow (5) differs from the $(\alpha, 1)$-flow by a constant $\alpha$. The flow (6) is a special case of the flow (5) with $\alpha=2$. The last flow (8) differs from the $(\alpha, \alpha+1)$-flow (7) by a constant $\alpha$.
\end{itemize}

\section{Basic properties of $(\alpha,\beta)$-flow}\label{section-convergence}
Since the two most representative flows (2) and (4) in Example \ref{Example-all-2d-flows} are have intensively been studied in \cite{GX2,GX3}, we mainly study 3-dimensional unified $(\alpha,\beta)$-flow in the remaining of this section. It's obviously to prove

\begin{proposition}\label{prop-hypersurface}
Let $\delta=\alpha-\beta+n-1$. If $\delta\neq0$, then $\sum_{i=1}^Nr_i^{\delta}(t)$ is invariant along the $(\alpha, \beta)$-flow. If $\delta=0$, then $\prod_{i=1}^Nr_i(t)$ is invariant along the flow.
\end{proposition}

\begin{theorem}\label{them-conveg-imply-solvable}
The critical points of the $(\alpha, \beta)$-flow (\ref{Def-unified-flow}) is a constant $\alpha$-curvature metric, and the solution $r(t)$ to this flow always exists locally. Moreover, if the solution $r(t)$ converges, then the combinatorial $\alpha$-Yamabe problem is solvable, that is,
there exists a constant $\alpha$-curvature metric in the combinatorial conformal class.
\end{theorem}
\begin{proof}
Note that, in $\mathfrak{M}_{\mathcal{T}}$, $K_i$ as a function of $r=(r_1,\cdots,r_N)^T$ is smooth and hence locally Lipschitz continuous. By Picard theorem in classical ODE theory, flow (\ref{Def-unified-flow}) has a unique solution $r(t)$, $t\in[0, \epsilon)$ for some $\epsilon>0$. The convergence of $r(t)$ means that there exists a metric $r^*\in\mathfrak{M}_{\mathcal{T}}$, such that $r(t)\rightarrow r^*$ according to the Euclidean topology. By the classical ODE theory, $r^*$ should be the critical point of flow (\ref{Def-unified-flow}), which implies the conclusion above.\qed
\end{proof}

It follows natural to know when the unified $(\alpha,\beta)$-flow converges. The following lemma is very useful:
\begin{lemma}\label{Lemma-3d-Lambda matrix}
(\cite{CR, G2, Ri})
Suppose $(M, \mathcal{T})$ is a triangulated 3-manifold with sphere packing metric $r$,
$\mathcal{S}=\sum K_ir_i$ is the Einstein-Hilbert-Regge functional. Then we have
\begin{equation}
\nabla_r\mathcal{S}=K.
\end{equation}
If we set
\begin{displaymath}
\Lambda=Hess_r\mathcal{S}=
\frac{\partial(K_{1},\cdots,K_{N})}{\partial(r_{1},\cdots,r_{N})}=
\left(
\begin{array}{ccccc}
 {\frac{\partial K_1}{\partial r_1}}& \cdot & \cdot & \cdot &  {\frac{\partial K_1}{\partial r_N}} \\
 \cdot & \cdot & \cdot & \cdot & \cdot \\
 \cdot & \cdot & \cdot & \cdot & \cdot \\
 \cdot & \cdot & \cdot & \cdot & \cdot \\
 {\frac{\partial K_N}{\partial r_1}}& \cdot & \cdot & \cdot &  {\frac{\partial K_N}{\partial r_N}}
\end{array}
\right),
\end{displaymath}
then $\Lambda$ is positive semi-definite with rank $N-1$ and
the kernel of $\Lambda$ is the linear space spanned by the vector $r$.
\end{lemma}

We recall the definition of $\alpha$-order combinatorial Laplacian.
\begin{definition}
(\cite{GX2,GX3}) Given a triangulated $n$-manifold $(M, \mathcal{T})$ with $n=2$ or $3$. For any $\alpha\in \mathds{R}$, the $\alpha$-order combinatorial Laplacian (``$\alpha$-Laplacian" for short) $\Delta_{\alpha}:C(\mathcal{T}_0)\rightarrow C(\mathcal{T}_0)$ is defined as
\begin{equation}\label{Def-alpha-Laplacian}
\Delta_{\alpha} f_i=\frac{1}{r_i^{\alpha}}\sum_{j\sim i}(-\frac{\partial K_i}{\partial r_j}r_j)(f_j-f_i)
\end{equation}
for $f\in C(\mathcal{T}_0)$.
\end{definition}
The $\alpha$-Laplacian (\ref{Def-alpha-Laplacian}) can also be written in a matrix form
\begin{equation}
\Delta_{\alpha}=-\Sigma^{-\alpha}\Lambda \Sigma
\end{equation}
with $\Delta_{\alpha} f=-\Sigma^{-\alpha}\Lambda \Sigma f$ for each $f\in C(\mathcal{T}_0)$, where $\Sigma=diag\big\{r_1,\cdots,r_N\big\}$.

\begin{theorem}\label{Thm-3d-alpha-beta-flow}
Given a triangulated manifold $(M^3, \mathcal{T})$. Suppose $r^*\in\mathbb{S}^{3}$ is a constant $\alpha$-curvature metric satisfying $\lambda_1(-\Delta_{\alpha})>\alpha s_{\alpha}^*$, or more specifically, $r^*\in\mathbb{S}^{3}$ is a constant $\alpha$-curvature metric with $\alpha s_{\alpha}^*\leq0$. If $\|r(0)-r^*\|$ is small enough, then the solution of the unified $(\alpha, \beta)$-flow (\ref{Def-unified-flow})
exists for all time and converges to $r^*$.\qed
\end{theorem}
\begin{proof}
Denote $\Gamma_i(r)=s_{\alpha}r_i^{\beta}-K_ir_i^{\beta-\alpha}$, $1\leq i\leq N$, then the $(\alpha, \beta)$-flow (\ref{Def-unified-flow}) can be written as $\dot{r}=\Gamma(r)$, which is an autonomous ODE system. Differentiating $\Gamma$ with respect to $r$, we can get
\begin{equation*}
D_r\Gamma(r)=\Sigma^{\beta-\alpha}\left(-\Lambda+\alpha s_{\alpha}\left(\Sigma^{\alpha-1}-\frac{r^\alpha (r^\alpha)^T}{\|r\|_{\alpha+1}^{\alpha+1}}\right)-H
\right),
\end{equation*}
where
\begin{equation*}
H=(\beta-\alpha)\Sigma^{-1}\left(\begin{array}{ccc}
                             K_1-s_{\alpha}r_1^{\alpha} &   &   \\
                               & \ddots &   \\
                               &   & K_N-s_{\alpha}r_N^{\alpha}\\
                      \end{array}
                \right)+\frac{r^{\alpha}\left(K-s_{\alpha}r^{\alpha}\right)^T}{\|r\|_{\alpha+1}^{\alpha+1}}.
\end{equation*}
At constant $\alpha$-curvature metric points $r^*$, $H=0$, thus we have
\begin{equation}\label{Diff-gamma at-r^*}
D_r\Gamma|_{r^*}=\Sigma^{\beta-\alpha}\left(-\Lambda+\alpha s_{\alpha}\left(\Sigma^{\alpha-1}-\frac{r^\alpha (r^\alpha)^T}{\|r\|_{\alpha+1}^{\alpha+1}}\right)\right)_{r^*}.
\end{equation}
We follow the tricks in \cite{GX1,GX3} to derive the conclusion. Denote $\widetilde{\Lambda}=\Sigma^{\frac{1-\alpha}{2}}\Lambda \Sigma^{\frac{1-\alpha}{2}}$. Then
$$-\Delta_{\alpha}=\Sigma^{-\alpha}\Lambda \Sigma=\Sigma^{-\frac{1+\alpha}{2}}\widetilde{\Lambda}\Sigma^{\frac{1+\alpha}{2}},$$
which implies that $$\lambda_1(-\Delta_{\alpha})=\lambda_1(\widetilde{\Lambda}).$$
Choose a matrix $P\in O(N)$ such that $P^T\widetilde{\Lambda}P=diag\{0,\lambda_1(\widetilde{\Lambda}),\cdots,\lambda_{N-1}(\widetilde{\Lambda})\}$.
Suppose $P=(e_0,e_1,\cdots,e_{N-1})$,
where $e_i$ is the $(i+1)$-column of $P$.
Then $\widetilde{\Lambda} e_0=0$ and $\widetilde{\Lambda} e_i=\lambda_i e_i,\,1\leq i\leq N-1$,
which implies $e_0=r^\frac{\alpha+1}{2}/\|r^\frac{\alpha+1}{2}\|$ and $r^\frac{\alpha+1}{2}\perp e_i,\,1\leq i\leq N-1$.
Hence $\left(I-\frac{r^{\frac{\alpha+1}{2}}(r^{\frac{\alpha+1}{2}})^T}{\|r\|^{\alpha+1}_{\alpha+1}}\right)e_0=0$ and $\left(I-\frac{r^{\frac{\alpha+1}{2}}(r^{\frac{\alpha+1}{2}})^T}{\|r\|^{\alpha+1}_{\alpha+1}}\right)e_i=e_i$, $1\leq i\leq N-1$.
Furthermore,
\begin{equation*}
\begin{aligned}
-D_r\Gamma|_{r^*}=&\Sigma^{\beta-\alpha}\Sigma^{\frac{\alpha-1}{2}}P diag\left\{0,\lambda_1(\widetilde{\Lambda})-\alpha s_{\alpha}^*,\cdots,\lambda_{N-1}(\widetilde{\Lambda})-\alpha s_{\alpha}^*\right\}P^T\Sigma^{\frac{\alpha-1}{2}}\\[4pt]
=&\Sigma^{\frac{\beta-\alpha}{2}}\Sigma^{\frac{\beta-1}{2}}P diag\left\{0,\lambda_1(\widetilde{\Lambda})-\alpha s_{\alpha}^*,\cdots,\lambda_{N-1}(\widetilde{\Lambda})-\alpha s_{\alpha}^*\right\}P^T\Sigma^{\frac{\beta-1}{2}}\Sigma^{-\frac{\beta-\alpha}{2}}\\[4pt]
\thicksim &\big(\Sigma^{\frac{\beta-1}{2}}P\big)diag\left\{0,\lambda_1(\widetilde{\Lambda})-\alpha s_{\alpha}^*,\cdots,\lambda_{N-1}(\widetilde{\Lambda})-\alpha s_{\alpha}^*\right\}\big(\Sigma^{\frac{\beta-1}{2}}P\big)^T.
\end{aligned}
\end{equation*}
This shows that the eigenvalue of $D_r\Gamma|_{r^*}$ are all negative when restricted to the hypersurface $\sum_{i=1}^Nr_i^{\delta}(t)$ when $\delta\neq0$, or to the hypersurface $\prod_{i=1}^Nr_i(t)$ when $\delta=0$. Then the theorem is a consequence of the Lyapunov Stability Theorem in classical ODE theory.
\end{proof}

\begin{remark}
The $\alpha$=0 case is of special interest. The unified $(0,\beta)$-flow always converges to a metric $r^*$ with Cooper and Rivin's curvature $K_i\equiv$ constant, if $\|r(0)-r^*\|$ is small enough. Specifically, there exists a $\epsilon>0$, such that if the initial metric $r(0)$ satisfies $\|r(0)-r^*\|<\epsilon$, then the solution $r(t)$ to the unified $(0,\beta)$-flow exists for all time $t\geq0$ and converges to $r^*$ as $t\rightarrow+\infty$. This means that constant $K$-curvature metric $r^*$ is locally stable. More specifically, it seems that the $(0,1)$-flow, which lies at the intersection of $(\alpha,1)$-flow and $(\alpha,\alpha+1)$-flow (see Figure \ref{fig1}), shows much better convergence properties than other types of unified $(\alpha,\beta)$-flows. For more details, see \cite{GJ,GM1,GM2}.
\end{remark}
\begin{figure}[!ht]
\begin{center}
\includegraphics[width=2.4in,height=1.8in]{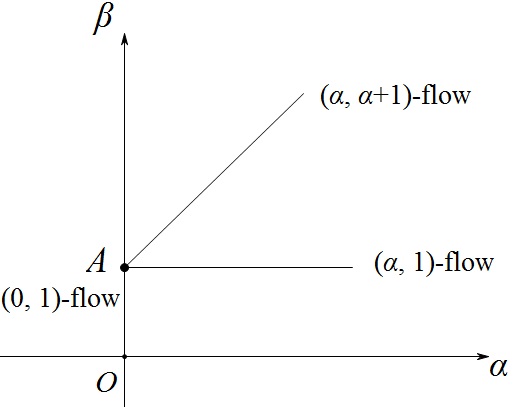}
\caption{($\alpha,\beta$) flows}
\label{fig1}
\end{center}
\end{figure}
\section{Calculating discrete constant $\alpha$-curvature metrics}
The $(\alpha,\beta)$-flow provides an efficient way to find constant $\alpha$-curvature metrics on triangulated manifolds. To see the power of this method, we take the 16-cell triangulation of $\mathbb{S}^3$ as an example, and see how the $(\alpha,\beta)$-flow can be used in concrete calculations.

Consider a standard geometric triangulation $\mathcal{T}_s$ of $\mathbb{S}^{3}$. Let $A_{1}=(1,0,0,0)$, $A_{2}=(-1,0,0,0)$, $B_{1}=(0,1,0,0)$, $B_{2}=(0,-1,0,0)$, $C_{1}=(0,0,1,0)$, $C_{2}=(0,0,-1,0)$, $D_{1}=(0,0,0,1)$, $D_{2}=(0,0,0,-1)$ be the vertices of $\mathcal{T}_s$, $P_{i}Q_{j}(\{P,Q\}\in\{A,B,C,D\},i,j=1,2)$
be the edges of $\mathcal{T}_s$, $P_{i}Q_{j}R_{k}$ $(\{P,Q,R\}\subset\{A,B,C,D\},i,j,k=1,2)$ be the faces of $\mathcal{T}_s$, and the regular tetrahedrons $A_{i}B_{j}C_{k}D_{l}(i,j,k,l=1,2)$ be the tetrahedrons of $\mathcal{T}_s$.

We then consider a topological triangulation $\mathcal{T}$ of $\mathbb{S}^{3}$, which has the same combinatorial structure with $\mathcal{T}_s$. $\mathcal{T}$ is often called the 16-cell triangulation of $\mathbb{S}^3$ in previous literature. It is easy to see that $\mathcal{T}$ carries a trivial constant $\alpha$-curvature metric for each $\alpha$. In fact, let $r_i=1$, then $r$ is exactly a constant $\alpha$-curvature metric. We want to know whether there are other (up to scaling) constant $\alpha$-curvature metrics? By evolving the unified $(\alpha,\beta)$-flow (\ref{Def-unified-flow}) with appropriate initial value, we can easily get many constant $\alpha$-curvature metrics (up to scaling). For example, when $\alpha=1$, we denote the metrics on $A_{1}$, $A_{2}$, $B_{1}$, $B_{2}$, $C_{1}$, $C_{2}$, $D_{1}$, $D_{2}$
by $r_{1}^{+},r_{1}^{-},r_{2}^{+},r_{2}^{-},r_{3}^{+},r_{3}^{-},r_{4}^{+},r_{4}^{-}$ respectively. Then
$$r_{1}^{\pm}  =r_{2}^{\pm}=\frac{1}{12}, \:\:r_{3}^{\pm} =r_{4}^{\pm}=\frac{1}{6}$$
is exactly a constant $1$-curvature metric. For this case,
\begin{align*}
K_{1}^{\pm} & =K_{2}^{\pm}=8\pi-16\arccos\frac{1}{\sqrt{10}},\\
K_{3}^{\pm} & =K_{4}^{\pm}=12\pi-8\arccos\frac{3}{5}-16\arccos\frac{1}{\sqrt{10}}.
\end{align*}
Furthermore, we have $\cfrac{K_{i}^{\pm}}{r_{i}^{\pm}}=12(8\pi-16\arccos\frac{1}{\sqrt{10}})\thickapprox -61.8$ for each $1\leq i \leq4$.\\[8pt]

\noindent
\textbf{Acknowledgements:} The first author would like to show his greatest respect to Professor Gang Tian for constant supports and encouragements. The research is supported by National Natural Science Foundation of China under grant (No.11501027), and Fundamental Research Funds for the Central Universities (Nos. 2015JBM103, 2014RC028 and 2016JBM071).

\footnotetext{{\it Huabin Ge}, Department of Mathematics, Beijing Jiaotong University, Beijing 100044, P.R. China, Email:hbge@bjtu.edu.cn.  {\it Ming Li}, Corresponding author. LSEC, ICMSEC, Academy of Mathematics and Systems Science, Chinese Academy of Sciences, Beijing 100190, China. Email: liming@lsec.cc.ac.cn}

\end{document}